\documentclass{amsart}
\usepackage{amsfonts}
\usepackage{color}
\usepackage{graphicx}

\newtheorem{thm}{Theorem}[section]

\newtheorem{lema}[thm]{Lemma}
\newtheorem{prop}[thm]{Proposition}
\theoremstyle{definition}
\newtheorem{defn}[thm]{Definition}
\theoremstyle{remark}

\newtheorem{rem}[thm]{Remark}
\numberwithin{equation}{section}
\newcommand{\R}{\mathbb R}
\newcommand{\ve}{\varepsilon}
\newcommand{\N}{\mathbb N}

\newcommand{\C}{\mathcal{C}}

\newcommand{\A}{\mathcal{A}}

\newcommand{\lam}{\lambda}
\newcommand{\Lam}{\Lambda}

\newcommand{\cd}{\rightharpoonup}

\newcommand{\cf}{\rightarrow}
\def\nequiv{\operatorname {/\!\!\!\!\!\equiv}}
\def\diver{\operatorname {\mathrm{div}}}

\begin{document}
\title[Quasilinear eigenvalues]{Quasilinear eigenvalues}
\author[J Fern\'andez Bonder, J P Pinasco, A M Salort]{Juli\'an Fern\'andez Bonder, Juan P. Pinasco, Ariel M. Salort }
\address{Departamento de Matem\'atica
 \hfill\break \indent FCEN - Universidad de Buenos Aires and
 \hfill\break \indent   IMAS - CONICET.
\hfill\break \indent Ciudad Universitaria, Pabell\'on I \hfill\break \indent   (1428)
Av. Cantilo s/n. \hfill\break \indent Buenos Aires, Argentina.}
\email[J. Fernandez Bonder]{jfbonder@dm.uba.ar}
\urladdr[J. Fernandez Bonder]{http://mate.dm.uba.ar/~jfbonder}
\email[J.P. Pinasco]{jpinasco@dm.uba.ar}
\urladdr[J.P. Pinasco]{http://mate.dm.uba.ar/~jpinasco}
\email[A.M. Salort]{asalort@dm.uba.ar}

%35P15      Estimation of eigenvalues, upper and lower bounds
%35P30      Nonlinear eigenvalue problems, nonlinear spectral theory
%35B27      Homogenization; equations in media with periodic structure

\keywords{Eigenvalue, homogenization, nonlinear eigenvalues \\
This work was partially supported by Universidad de Buenos Aires under grant 20020100100400 and by CONICET (Argentina) PIP 5478/1438.
}

\begin{abstract}
In this work, we review and extend some well known results for the eigenvalues of the
Dirichlet $p-$Laplace operator to a more general class of monotone quasilinear elliptic
operators. As an application we obtain some homogenization results for nonlinear
eigenvalues.
\end{abstract}

\maketitle

%*************************************************************
%*****************************************************
\section{Introduction}

In this work we review the eigenvalue problem associated to the $p-$Laplace operator,
$$
\begin{cases}
 -\Delta_p := -\diver(|\nabla u|^{p-2}\nabla u) = \lambda \rho u  & \text{in }\Omega\\
u=0 & \text{on } \partial\Omega,
\end{cases}
$$
we describe its history and the main results obtained in the past years, and we extend those
results to more general quasilinear problems.

 To be precise, we consider the equation
\begin{equation}\label{1.1}
\begin{cases}
-\diver(a(x,\nabla u)) = \lambda \rho(x) |u|^{p-2} u & \text{in }\Omega\\
u=0 & \text{on } \partial\Omega
\end{cases}
\end{equation}
where the functions $a(x,\xi)$ has the same homogeneity of $|\xi|^{p-2}\xi$,  and has
precise hypotheses that we state below (see Section 3). The domain $\Omega\subset \R^N$
is assumed to be bounded, $N\ge 1$, and the weight function $\rho$ is assumed to be
bounded away from zero and infinity.

We denote the spectrum of \eqref{1.1} by $\Sigma$, i.e.
$$
\Sigma := \{\lambda\in\R\colon \text{there exists a nontrivial weak solution to \eqref{1.1}}\},
$$
and we focus our attention on the properties of the set $\Sigma$ and the associated
eigenfunctions.

The paper is organized as follows: in Section 2 we introduce the  origins of the $p-$Laplace
operator, and the history of the developments made for the eigenvalue problem. In Section 3
we introduce more general operators generalizing the $p-$Laplacian, we define its
variational  spectrum (which is not known if coincides with $\Sigma$), and we collect some
necessary definitions and results. Section 4 is devoted to the properties of eigenvalues and
eigenfunctions. Finally we close the paper with some recent results on eigenvalue
homogenization in Section 5.

\section{A bit of history}

The one dimensional $p-$Laplace ordinary differential equation,
\begin{equation}\label{unad}
 -(|y'|^{p-2}y')' =\rho(x) |y|^{p-2}y
\end{equation}
was studied first by Leonhard Euler, in the work \cite{Euler} appeared in 1728. Several cases
were presented as an example of a nonlinear second order equation which cannot be
integrated with known techniques.

He considered nonlinear equations of the general form
 $$
a x^m dx^p = y^n dy^{p-2}ddy,
$$
where $a$ is a constant,  which correspond to equation \eqref{unad} when $a=-(p-1)^{-1}$,
$\rho= x^m$, and $n=1-p$. He introduced in that work the exponential function in order to
change variables, and reduced it to a first order equation. He used the following substitution
$$
\begin{cases}
x^{\frac{1}{n+p-1}} = e^{\int z dt} \\
y=e^{(m+p)\int z dt}
\end{cases}
$$
in paragraphs 7-9, where ``\emph{dx constant ponatur}" means that $x$ was chosen as the
independent variable, and then $ddx=0$.

Observe that, although this substitution enable us to work with Emden-Fowler like equations
$$
-(|y'|^{p-2}y')'=|y|^{q-2}y,
$$
 a different one is needed when $p=q$ since $n=1-p$. This case was included in paragraph 20, where he considered
  $$
dy^{m-1} ddy = P(x)y^m dx^{m+1}+Q(x)y^{m-b}dy^b dx^{m-b+1}
$$
where we have interchanged $x$ and $y$ for readability. In modern notation, with $m=p-1$,
reads
$$
|y'|^{p-2} y'' = P(x)|y|^{p-2}y +Q(x)y^{p-1-b}|y'|^b.
$$
Euler emphasized the homogeneity of the three terms involved, and the fact that more similar
terms can be added. For this equation he derived a generalized Riccati equation, rediscoverd
for the one dimensional $p-$Laplacian in the $20^{th}$ century:
$$
z^{p-2}z'+z^{p} = Q(y)z^{b} + P(y),
$$
or, by calling $z^{p-1}=w$
\begin{equation}\label{riccati}
\frac{w'}{p-1} +w^{\frac{p}{p-1}}  = Q(y)w^{\frac{b}{p-1}}  + P(y).
\end{equation}

When $Q\equiv 0$, the Riccati equation \eqref{riccati} was used by Beesack in 1961, see
\cite{Bees}, connected with optimal constants in Hardy's inequality. Let us remark that Bihari
in 1956 studied a related nonlinear equation in \cite{Bih},
$$
-y'' =  Q(x)f(y,y'),
$$
with $y f(y,y')>0$, $f(cy,cy')= cf(y,y')$, and  Lipschitz on every bounded domain of $ \R$.
However, this last condition excludes the $p-$Laplacian.

Few years later, Browder studied $N-$dimensional quasilinear equations, inspired in previous
works of Vi{\v{s}}ik, see \cite{Bro1, Bro2} and the references in this work.  He introduced the
so-called monotonicity methods (discovered almost simultaneously by Minty \cite{Min}, and
Vainberg and Kachurovski \cite{VK}). Since then, the study of quasilinear operators
experimented an explosive growth, and both variational and non-variational techniques were
introduced by by Browder, Fu{\v{c}{\'{i}k, Ladyzhenskaya, Leray, Lions, Morrey, Ne{\v{c}}as,
Rabinowicz, Schauder, Serrin, Trudinger... among several other mathematicians.

\bigskip
The eigenvalue problem for the $p-$Laplace operator started with the pioneering work of Browder \cite{Browder1, Browder2, Browder3, Browder4}. In those papers, he studied the nonlinear eigenvalue problem $A(u) = \lambda B(u)$, where $\lambda$ is real parameter, and
 $$
 A(u)= \sum_{|\alpha|\le m}(-1)^{|\alpha|} D^{\alpha}F_{p^{\alpha}}(x, u, \dots, D^m u),
 $$
 $$
B(u)= \sum_{|\beta|\le m-1}(-1)^{|\beta|} G_{p^{\beta}}(x, u, \dots, D^{m-1} u).
 $$
 This higher-order elliptic problem is the Euler-Lagrange equation of the variational problem
 $$
\min\left\{\int_{\Omega} F(x, u, \dots, D^m u)\, dx \colon u\in V \mbox{ and } \int_{\Omega}G(x,u,\dots, D^{m-1}u)\, dx= c\right\}
 $$
 where
$V$ is some closed subspace of  $W^{m,p} (\Omega)$.

By introducing the  variables $\zeta = \{\zeta_{\alpha} \colon |\alpha|=m\}$, $\psi= \{\psi_{\xi} \colon |\xi|\le m-1\}$,  the functions $F$, $G$ are measurable in $x$, and $C^1$ in the variables $\psi$, $\zeta$,   satisfying polynomial growth conditions which include  the following particular case for the Dirichlet boundary value problem,
\begin{align*}
|F(x,\psi, \zeta)| \le & c ( 1 + |\zeta|^p + |\psi|^p), \\
|F_{\alpha}|+|F_{\xi}|  \le & c (1+ |\zeta|^{p-1} + |\psi|^{p-1}),\\
G =& u^q
\end{align*}
where $1<p<\infty$, $q< np(n-mp)^{-1}$ for $n>mp$, and any $q$ for $n<mp$.

With appropriate conditions of ellipticity and coercivity, the existence of an eigenvalue and a corresponding eigenfunction which is a weak solution of $A(u)=\lambda B(u)$ can be found in \cite{Browder1}. Moreover, for $p\ge 2$, and imposing more regularity on $F$ and $G$ (at least $C^2$ in the variables $\psi$ and $\zeta$), the existence of  a sequence of eigenvalues was announced in \cite{Browder2} and proved in \cite{Browder3}. We can found in those works the heavy --now standard-- machinery of Palais-Smale sequences, deformation lemmas, Lyusternik-Schnirelman category, and monotonicity arguments.

Finally, a different approach can be found in \cite{Browder4}, based on Galerkin approximations. Here, for higher-order quasilinear operators satisfying the same coercivity and polynomial growth conditions, the regularity conditions can be relaxed, and a sequence of eigenvalues is obtained for $C^1$ functions and $1<p<\infty$.

Since then, several works devoted to this subject appeared. The interested reader can browse into the book of Fu{\v{c}}{\'{\i}}k, Ne{\v{c}}as, J. Sou{\v{c}}ek, and V. Sou{\v{c}}ek  \cite{Fucik} for a survey up to the mid 1970s. It is worth noticing that several of the works cited therein were published in Russian, or in journals from Central and East Europe, so many results were rediscovered later. Nonlinear eigenvalue problems was an active area of research among Czech, German and Hungarian mathematicians in this decade (we mention Amann, Elbert, Fu{\v{c}}{\'{\i}}k, Hess, Kufner, Ne{\v{c}}as, and Zeidler, to cite only a few of them).  See for example \cite{Am} for generalizations of the Browder's results and applications to Hammerstein's equations; \cite{fuck} for integro-differential equations; and \cite{Zeid} where two sequences of eigenvalues going to $\pm \infty$ were obtained for indefinite eigenvalue problems.

In the $p-$Laplacian case, i.e. when $a(x,\xi)=|\xi|^{p-2}\xi$, and for Dirichlet boundary conditions,  the structure of $\Sigma$ has been analyzed by several authors and it is know that
\begin{itemize}
\item $\Sigma\subset (0,\infty)$ is a closed set, see the work of Lindqvist \cite{Lind}.

\item $\lambda_1=\min\Sigma$ is the only eigenvalue that has a nonnegative associated eigenfunction (i.e., is a principal eigenvalue). This principal eigenvalue has a variational characterization given by
$$
\lambda_1 = \inf_{v\in W^{1,p}_0(\Omega)} \frac{\int_\Omega |\nabla v|^p\, dx}{\int_\Omega |v|^p\, dx}.
$$
The above infimum is realized precisely at eigenfunctions associated to $\lambda_1$. See \cite{Anane, Lind}.

\item $\lambda_1$ is {\em isolated} and {\em simple}. That is, there exists $\delta>0$ such
    that $$(\lambda_1,\lambda_1+\delta)\cap \Sigma = \emptyset,$$ and if $u_1, u_2\in
    W^{1,p}_0(\Omega)$ are two eigenfunctions associated to $\lambda_1$ then there
    exists $c\in \R$ such that $u_1 = c u_2$.  See \cite{AlHu, Anane, Lind}.

\item There exists a sequence of \emph{variational eigenvalues}, usually denoted by $\Sigma_{\text{var}}$ given by
$$
\lambda_k := \inf_{C\in \C_k} \sup_{v\in C} \frac{\int_\Omega |\nabla v|^p\, dx}{\int_\Omega |v|^p\, dx},
$$
where $\C_k :=\{ C\subset W^{1,p}_0(\Omega)\colon C \text{ is closed}, C=-C,
 \gamma(C)\ge k\}$ and $\gamma$ is the Krasnoselskii genus. This was the approach of
 Browder, by using the Lyusternik-Schnirelmann theory, see also \cite{GAP0, GAP}.

\item There exists other possible ways to construct variational eigenvalues for this type of equations. Some authors prefer to call $\Sigma_\text{var}$ the {\em Lyusternik-Schnirelmann} eigenvalues, although in this work we will use the more extended denomination and refer to these as the {\em variational eigenvalues}. See \cite{dra} for a comprehensive discussion on this topic.

\item The sequence $\Sigma_{\text{var}}$ has the asymptotic behavior given by the Weyl's law
$$
c \left(\frac{k}{|\Omega|}\right)^{\frac{N}{p}}\le \lambda_k\le C \left(\frac{k}{|\Omega|}\right)^{\frac{N}{p}},
$$
for some (universal) constants $c,C>0$ depending only on $N$ and $p$. See \cite{Friedlander, GAP}.

\item As the first eigenvalue $\lambda_1$ is isolated in $\Sigma$ which is a closed set, the second eigenvalue is well defined as
$$
\Lambda_2 = \min\{\lambda\in\Sigma\colon \lambda>\lambda_1\}>\lambda_1.
$$
It is known that $\Lambda_2$ coincides with the second variational eigenvalue $\lambda_2$. See \cite{AT, DFCG}.

\item For one dimensional problems, $\Omega=(a,b)\subset \R$ it is known that any
    eigenvalue is simple, the eigenfunction corresponding to $\lambda_k$ has exactly $k+1$
    zeros counting the boundary points $a$ and $b$, and this fact enable us to obtain them
    variationally. The eigenvalues can be computed explicitly, and the corresponding
    eigenfunctions are obtained in terms of the Gaussian hypergeometric function  (see
    \cite{Bai, DDM, FBP, Walter}).

\item A major open question is to know whether if $\Sigma = \Sigma_{\text{var}}$ or not. An
    answer to this problem is only known in one space dimension. In this situation the
    question is answered positively, using that eigenvalues associated to $\lambda_k$ has
    $k$ nodal domains. See \cite{FBP, Walter}. A negative result is known for periodic
    boundary conditions, see \cite{biry, dra, drta}.
\end{itemize}

The objective of this paper is the extension of all these facts to the more general problem
\eqref{1.1}. Let us observe that the first item follows by monotonicity arguments, and the second one was already generalized to \eqref{1.1} by
\cite{Kaw-Lu-Pra}. So here we complete the program in performing the others extensions.

As a corollary of our results we obtain some alternative proofs of convergence theorems for
nonlinear eigenvalue homogenization that were originally proved in \cite{Ch-dP}.

\section{Preliminary Results}
In this section we review some results gathered from the literature, enabling us to clearly state our results and making the paper self-contained.

\subsection{Monotone operators}

First, we give the precise hypotheses on the coefficient $a(x,\xi)$ in order to be able to treat the eigenvalue equation \eqref{1.1} variationaly. The precise context is the assumption that the induced operator $\mathcal{A}\colon W_0^{1,p}(\Omega)\to W^{-1,p'}(\Omega)$ given by
$$
\mathcal{A} u := -\diver(a(x,\nabla u)),
$$
defines a monotone operator.

So we assume that $a\colon \Omega\times \R^N\to \R^N$ satisfies,  for every $\xi\in\R^N$ and  a.e. $x\in \Omega$, the following conditions:

\begin{enumerate}
\item[(H0)] {\em measurability:} $a(\cdot,\cdot)$ is a Carath\'eodory function, i.e. $a(x,\cdot)$ is continuous a.e. $x\in \Omega$,  and $a(\cdot,\xi)$ is measurable for every $\xi\in\R^N$.

\item[(H1)] {\em monotonicity:} $0\le (a(x,\xi_1)-a(x,\xi_2))(\xi_1-\xi_2)$.

\item[(H2)] {\em coercivity:} $\alpha |\xi|^p \le a(x,\xi) \xi$.

\item[(H3)] {\em continuity:} $a(x,\xi)\le \beta|\xi|^{p-1}$.

\item[(H4)] {\em $p-$homogeneity:} $a(x,t\xi)=t^{p-1} a(x,\xi)$ for every $t>0$.

\item[(H5)] {\em oddness:} $a(x,-\xi) = -a(x,\xi)$.
\end{enumerate}

Let us introduce $\Psi(x,\xi_1, \xi_2)=a(x,\xi_1) \xi_1 +a(x,\xi_2) \xi_2$ for all $\xi_1, \xi_2 \in \R^N$, and all $x\in \Omega$; and let $\delta=min\{p/2, (p-1)\}$.

\begin{enumerate}
\item[(H6)] {\em equi-continuity:}
$$
|a(x,\xi_1) -a(x,\xi_2)| \le c \Psi(x,\xi_1, \xi_2)^{(p-1-\delta)/p}(a(x,\xi_1) -a(x,\xi_2)) (\xi_1-\xi_2)^{\delta/p}
$$

\item[(H7)] {\em cyclical monotonicity:} $\sum_{i=1}^k  a(x,\xi_i) (\xi_{i+1}-\xi_i) \le 0$, for all $k\ge 1$, and $\xi_1,\ldots, \xi_{k+1}$, with $\xi_1=\xi_{k+1}$.

\item[(H8)] {\em strict monotonicity:} let $\gamma = \max(2,p)$, then
$$
\alpha |\xi_1-\xi_2|^{\gamma}\Psi(x,\xi_1,\xi_2)^{1-(\gamma/p)}\le (a(x,\xi_1)-a(x,\xi_2)) (\xi_1-\xi_2).
$$
\end{enumerate}

See \cite{Con}, Section 3.4 where a detailed discussion on the relation and implications of every condition (H0)--(H8) is given.

In particular, under these conditions, we have the following Proposition:

\begin{prop}[\cite{Con}, Lemma 3.3]\label{potential.f}
Given $a(x,\xi)$ satisfying {\em (H0)--(H8)} there exists a unique Carath\'eodory function $\Phi$ which is even, $p-$homogeneous strictly convex and differentiable in the variable $\xi$ satisfying
\begin{equation}\label{cont.coer.phi}
\alpha |\xi|^p \le \Phi(x,\xi)\le \beta |\xi|^p
\end{equation}
for all $\xi\in \R^N$ a.e. $x\in\Omega$ such that
$$
\nabla_{\xi} \Phi(x,\xi)= p\, a(x,\xi)
$$
and normalized such that $\Phi(x,0)=0$.
\end{prop}

\begin{rem} In the one dimensional case, hypotheses (H4) and (H5) imply that
$$
a(x,\xi) = a(x) |\xi|^{p-2}\xi,
$$
with $a(x) := a(x,1)$. In this case, the potential function $\Phi$ is given by
$$
\Phi(x,\xi) = a(x) |\xi|^p.
$$
\end{rem}

\begin{rem}
In dimension $N>1$, the prototypical example for $a(x,\xi)$ is
$$
a(x,\xi) = |A(x)\xi\cdot\xi|^{\frac{p-2}{2}} A(x)\xi.
$$
In this case, the potential function $\Phi(x,\xi)$ of Proposition \ref{potential.f} is given by
$$
\Phi(x,\xi) = 2|A(x)\xi\cdot\xi|^{\frac{p}{2}}
$$
\end{rem}

\subsection{Definition of $G$-convergence}\label{gammaconv}

For our application to homogeneization, the concept of $G-$convergence of operators is needed. We review here the basic definitions and properties.

\begin{defn}
We say that the family of operators $\mathcal{A}_\ve u := -\diver(a_\ve(x,\nabla u))$ $G$-converges to $\mathcal{A}u:=-\diver(a(x,\nabla u))$ if for every $f\in W^{-1,p'}(\Omega)$  and for every $f_\ve$ strongly convergent to $f$ in $W^{-1,p'}(\Omega)$, the solutions $u^\ve$ of the problem
\begin{equation*}
\begin{cases}
-\diver(a_\ve(x,\nabla u^\ve))=f_\ve &\textrm{ in } \Omega \\
u^\ve=0 & \textrm{ on } \partial \Omega
\end{cases}
\end{equation*}
satisfy the following conditions
\begin{align*}
 u^\ve \cd u  &\qquad \mbox{ weakly in }  W^{1,p}_0(\Omega), \\
 a_\ve(x, \nabla u^\ve) \cd a(x,\nabla u)  &\qquad \mbox{ weakly in }  (L^{p}(\Omega))^N,
\end{align*}
where $u$ is the solution to the equation
\begin{equation*}
\begin{cases}
-\diver(a(x,\nabla u))=f & \textrm{ in } \Omega  \\
u=0 &\textrm{ on } \partial \Omega.
\end{cases}
\end{equation*}
\end{defn}

For instance, in the linear periodic case,  the family $-\diver(A(\tfrac{x}{\ve})\nabla u)$ $G$-converges to a limit operator $-\diver(A^*\nabla u)$ where $A^*$ is a constant matrix which can be
characterized in terms of $A$ and certain auxiliary functions. See for example \cite{Cio}.

\medskip

It is shown in \cite{Con} that properties (H0)--(H8) are stable under $G-$convergence, i.e.
\begin{thm}[\cite{Con}, Theorem 2.3]
If $\mathcal{A}_\ve u := -\diver(a_\ve(x,\nabla u))$ $G-$converges to $\mathcal{A}u:=-\diver(a(x,\nabla u))$ and $a_\ve(x,\xi)$ satisfies {\em (H0)--(H8)} uniformly, then $a(x,\xi)$ also satisfies {\em (H0)--(H8)}.
\end{thm}

In the periodic case, i.e.  when $a_\ve(x,\xi)=a(\tfrac{x}{\ve}, \xi)$, and $a(\cdot, \xi)$ is $Q-$periodic for every $\xi\in \R^N$, one has that $\mathcal{A_\ve}$ $G-$converges to the homogenized operator $\mathcal{A}_h$ given by $\mathcal{A}_h u = -\diver(a_h(\nabla u))$, where  $a_h:\R^N \cf \R^N$ can be characterized by
\begin{equation}\label{ah}
 a_h(\xi)=\lim_{s\to \infty} \frac{1}{s^N} \int_{Q_s(z_s)} a(x,\nabla \chi^\xi_s + \xi)dx
\end{equation}
 where $\xi\in \R^N$, $Q_s(z_s)$ is the cube of side length $s$ centered at $z_s$  for any family $\{z_s\}_{s>0}$ in $\R^N$, and $\chi^\xi_s$ is the solution of the following auxiliary problem
\begin{equation}
\begin{cases}
 -\diver(a(x,\nabla \chi^\xi_s + \xi))=0 \quad \textrm{ in } Q_s(z_s) \\
 \chi^\xi_s\in W^{1,p}_0(Q_s(z)),
\end{cases}
\end{equation}
see \cite{DF1992} for the proof.

In the general case, one has the following compactness result due to \cite{Chi1}
\begin{prop}[\cite{Chi1}, Theorem 4.1]\label{Gconv.ve}
Assume that $a_\ve(x,\xi)$ satisfies {\em (H1)--(H3)} then, up to a subsequence, $\mathcal{A}_\ve$ $G-$converges to a maximal monotone operator $\mathcal{A}$ whose coefficient $a(x,\xi)$ also satisfies {\em (H1)--(H3)}
\end{prop}

In the one dimensional setting the $G-$limit is easily computed. In fact we have the following
fairly easy proposition. For $p=2$ this is well known, see \cite{Allaire-book} and for general
$p$ the extension is straightforward

\begin{prop}\label{G.conver.perio}
Let $\A_\ve u := -(a_\ve(x)|u'|^{p-2}u')'$ with $a_\ve \in L^\infty(\R)$ that satisfies
\begin{equation}\label{cota.a}
\alpha \le a_\ve(x)\le \beta,
\end{equation}
for some constants $\alpha,\beta>0$. Then, up to a subsequence, $\A_\ve$ $G-$converges to $\A u := -(a_p^*(x)|u'|^{p-2}u')'$, with $a_p^*\in L^{\infty}(\R)$ given by
$$
a_p^* = \bar a_p^{-(p-1)} \qquad \text{and}\qquad a_\ve^{-\frac{1}{p-1}} \stackrel{*}{\rightharpoonup} \bar a_p.
$$
\end{prop}

\begin{proof}
Let $f_\ve\in W^{-1,p'}(I)$ be such that $f_\ve\to f$ in $W^{-1,p'}(I)$.

Let  $g_\ve\in L^p(I)$ be such that $g'_\ve = f_\ve$ and $g_\ve \to g$ in $L^p(I)$. Hence $g'=f$.

Let $u_\ve$ be the weak solution to
$$
\begin{cases}
-(a_{\ve}(x)|u_\ve'|^{p-2}u_\ve')' = f_\ve & \mbox{in } I\\
u_\ve(0)=u_\ve(1)=0
\end{cases}
$$
Then, there exists a constant $c_\ve$ such that $a_\ve(x)|u_\ve'|^{p-2}u_\ve' = c_\ve - g_\ve$.

Let $\varphi_p(x) = |x|^{p-2}x$. Then $\varphi_p$ is invertible and so
\begin{equation}\label{uprima}
u_\ve' = \varphi_p^{-1}(c_\ve-g_\ve) a_\ve(x)^{-\frac{1}{p-1}}.
\end{equation}
Since $(u_\ve)_{\ve>0}$ is bounded in $W^{1,p}_0(I)$, we can assume that is weakly convergent to some $u\in W^{1,p}_0(I)$ and, since $a_\ve$ is bounded away from zero and infinity so is $a_\ve^{-\frac{1}{p-1}}$, so we can assume that there exists $\bar a_p\in L^\infty(I)$ such that
$$
a_\ve^{-\frac{1}{p-1}} \stackrel{*}{\rightharpoonup} \bar a_p.
$$
Moreover, we can assume that $g_\ve\to g$ in $L^p(I)$, and that $c_\ve\to c$.

Now we can pass to the limit in \eqref{uprima} and obtain
$$
u' = \varphi_p^{-1}(c-g)\bar a_p(x)
$$
The proof is now complete.
\end{proof}

%%%%%%%%%%%%%%%%%%%%%%%%%%

\section{Properties of the eigenvalues and eigenfunctions}

In this section we prove the main results of the paper, namely we study the properties of the spectrum $\Sigma$ of the following (nonlinear) eigenvalue problem
\begin{equation} \label{eps1}
\begin{cases}
-\diver(a(x,\nabla u) ) = \lam \rho |u|^{p-2}u &\quad \textrm{ in } \Omega \\
u = 0  &\quad \textrm{ on } \partial \Omega \\
\end{cases}
\end{equation}
where $a(x,\xi)$ verifies (H0)--(H8) and
\begin{equation}\label{cota.rho}
0<\rho^-\le \rho(x)\le \rho^+<\infty\qquad \mbox{ a.e. in } \Omega.
\end{equation}
As we mentioned in the introduction we extend here to \eqref{eps1} the results that are well-known for the $p-$Laplacian case.

The methods in the proofs here very much resembles the ones used for the $p-$Laplacian
and we refer the reader to the articles \cite{ACM, AT, Anane, KawL, Lind}.

\medskip

We recall that the spectrum $\Sigma$ is defined by
$$
\Sigma := \{\lam\in\R\colon \mbox{there exists } u\in W^{1,p}_0, \mbox{ nontrivial solution to \eqref{eps1}}\}.
$$

We begin with this proposition
\begin{prop}\label{Sigma.cerrado}
The spectrum $\Sigma$ of \eqref{eps1} is closed and, moreover, $\Sigma \subset (0,\infty)$.
\end{prop}

\begin{proof}
First, observe that (H2) trivially implies that $\Sigma\subset (0,+\infty)$. In fact, if $\lambda\in\Sigma$ and $u\in W^{1,p}_0(\Omega)$ is an eigenfunction associated to $\lambda$, then we have, from (H2)
\begin{equation}\label{facil}
\alpha \int_\Omega |\nabla u|^p\, dx \le \int_{\Omega} a(x,\nabla u)\nabla u\, dx = \lambda \int_\Omega \rho(x) |u|^p\, dx,
\end{equation}
from where it follows that
$$
\lambda \ge \frac{\alpha\int_\Omega |\nabla u|^p\, dx}{\int_\Omega \rho(x) |u|^p\, dx}>0.
$$

The fact that $\Sigma$ is closed follows from the monotonicity of the operator $\A$. In fact, let $\lambda_j\in \Sigma$ be such that $\lam_j\to\lam$ and let $u_j\in W^{1,p}_0(\Omega)$ be an eigenfunction associated to $\lam_j$. We can assume, from (H4), that $u_j$ is chosen so that  $\|u_j\|_{L^p(\Omega)}=1$. Then, since $\{\lam_j\}_{j\in\N}$ is bounded, from \eqref{facil} it follows that $\{u_j\}_{j\in\N}$ is bounded in $W^{1,p}_0(\Omega)$.

Therefore, taking a subsequence if necessary, we have that there exists $u\in W^{1,p}_0(\Omega)$ and, from (H3) that there exists $\eta\in (L^{p'}(\Omega))^N$ such that
\begin{align*}
u_j\rightharpoonup u &\quad \text{weakly in } W^{1,p}_0(\Omega)\\
u_j\to u & \quad \text{strongly in } L^p(\Omega) \text{ and a.e. in }\Omega\\
a(x,\nabla u_j) \rightharpoonup \eta & \quad \text{weakly in } L^{p'}(\Omega)
\end{align*}
From these convergences we obtain that $\|u\|_{L^p(\Omega)}=1$ (so that $u\neq 0$) and
\begin{equation}\label{eq.eta}
\int_\Omega \eta\nabla v\, dx = \lambda \int_\Omega \rho(x) |u|^{p-1}u v\, dx
\end{equation}
for every $v\in W^{1,p}_0(\Omega)$. So, the proof will be finished if we show that
\begin{equation}\label{eq.eta2}
\int_\Omega a(x,\nabla u)\nabla v\, dx = \int_\Omega \eta\nabla v\, dx
\end{equation}
for every $v\in W^{1,p}_0(\Omega)$. For this purpose, we make use of the monotonicity inequality (H1) and the fact that $u_j$ is an eigenfunction associated to $\lambda_j$. In fact, for every $w\in W^{1,p}_0(\Omega)$,
\begin{align*}
0&\le \int_\Omega (a(x,\nabla u_j) - a(x,\nabla w))(\nabla u_j-\nabla w)\, dx\\
& = \int_{\Omega}\lam_j\rho(x) |u_j|^{p-2}u_j (u_j-w)\, dx - \int_\Omega a(x,\nabla w)(\nabla u_j - \nabla w)\, dx.
\end{align*}
Taking the limit $j\to\infty$ in the former inequality, we get, using \eqref{eq.eta},
\begin{align*}
0&\le \int_{\Omega}\lam \rho(x) |u|^{p-2}u (u-w)\, dx - \int_\Omega a(x,\nabla w)(\nabla u - \nabla w)\, dx\\
& = \int_\Omega \eta(\nabla u - \nabla w)\, dx  - \int_\Omega a(x,\nabla w)(\nabla u - \nabla w)\, dx.
\end{align*}
So, if we take $w=u-tv$ with $v\in W^{1,p}_0(\Omega)$ given and $t>0$, we immediately get
$$
0\le \int_\Omega (\eta - a(x,\nabla u - t\nabla v))\nabla v\, dx,
$$
and taking $t\to 0+$, we arrive at
$$
0\le \int_\Omega (\eta - a(x,\nabla u))\nabla v\, dx.
$$
From this inequality is easy to see that \eqref{eq.eta2} holds and so the claim follows.
\end{proof}

The existence of a sequence of variational eigenvalues for \eqref{eps1} can be traced back to
the papers of F. Browder, as we pointed out before. We state the result here for further
reference.

\begin{thm}\label{teo.lamk}
Let $\{\lam_k\}_{k\in\N}$ be the sequence defined by
$$
\lam_k = \inf_{C\in \C_k} \sup_{v \in C} \frac{\int_\Omega  \Phi(x,\nabla v)}{\int_\Omega \rho |v|^p}
$$
where $\Phi(x,\xi)$ is the potential function given in Proposition \ref{potential.f},
$$
\C_k=\{C\subset W^{1,p}_0(\Omega) : C \textrm{ closed, } C=-C, \,  \, \gamma(C)\geq k\}
$$
and $\gamma(C)$ is the Kranoselskii genus.

Then $\{\lam_k\}_{k\in\N}\subset \Sigma$ and $\lam_k\to\infty$ as $k\to\infty$.
\end{thm}

We refer the reader to \cite{Rab2} for the definition and properties of $\gamma$.

As for the asymptotic behavior of the sequence $\Sigma_{\text{var}} = \{\lambda_k\}_{k\in\N}$ this follows easily from the variational characterization given in Theorem \ref{teo.lamk}, the coercivity inequality \eqref{cont.coer.phi} and the asymptotic behaviors for the eigenvalues of the $p-$Laplacian found in \cite{GAP} and refined in \cite{Friedlander}.

More precisely we have
\begin{thm}
There exists $c,C>0$ depending only on $p,N$ such that
$$
c\frac{\alpha}{\rho^+}\left(\frac{k}{|\Omega|}\right)^{\frac{p}{N}}\le \lam_k\le C\frac{\beta}{\rho^-}\left(\frac{k}{|\Omega|}\right)^{\frac{p}{N}},
$$
where $\alpha,\beta$ are given in \eqref{cont.coer.phi} and $\rho_-, \rho_+$ are given in \eqref{cota.rho}.
\end{thm}

\begin{proof}
From \eqref{cont.coer.phi} and \eqref{cota.rho} it follows that, for every $v\in W^{1,p}_0(\Omega)$ we have
$$
\frac{\alpha}{\rho^+}\frac{\int_\Omega |\nabla v|^p\, dx}{\int_\Omega |v|^p\, dx}\le  \frac{\int_\Omega  \Phi(x,\nabla v)}{\int_\Omega \rho |v|^p} \le \frac{\beta}{\rho^-}\frac{\int_\Omega |\nabla v|^p\, dx}{\int_\Omega |v|^p\, dx}.
$$
From these inequalities and the variational characterization of $\Sigma_{\text{var}}$ we obtain
$$
\frac{\alpha}{\rho^+}\mu_k\le \lam_k\le \frac{\beta}{\rho^-}\mu_k,
$$
where $\{\mu_k\}_{k\in\N}$ are the variational eigenvalues of the $p-$Laplacian. Now, the conclusion of the Theorem follows from the Weyl's asymptotic formula for $\{\mu_k\}_{k\in\N}$ proved in \cite{Friedlander}.
\end{proof}

The following maximum principle for quasilinear operators was proved in \cite{Kaw-Lu-Pra}
and it will be most useful in the sequel.
\medskip
\begin{thm}[\cite{Kaw-Lu-Pra}, Section 6.2]\label{SMP}
Assume that $u\in W^{1,p}_{\text{loc}}(\Omega)$ satisfies
$$
\int_\Omega a(x,\nabla u)\nabla \phi -\rho|u|^{p-2}u \phi \geq 0,\quad \forall \phi \in C_0^\infty(\Omega),\ \phi \geq 0.
$$
Consider its zero set
$$
\mathfrak{Z}:=\{x\in \Omega\colon \tilde u(x)=0\},
$$
where $\tilde u$ is the $p-$quasi continuous representative of $u$.

Then, either $Cap_p(\mathfrak{Z})=0$ or $u=0$.
\end{thm}

For the properties of the $p-$capacity and the $p-$quasi continuous representative of a Sobolev functions, we refer to \cite{Evans-Gariepy}.

\bigskip
The following result gives the positivity of the first eigenfunction.
\begin{thm}[\cite{Kaw-Lu-Pra}, Proposition 5.3] Let $u$ be an eigenfunction corresponding to $\lam_1$. Then exactly one of the following alternative holds:
$$u>0  \qquad or \qquad u<0$$
and the set of zeroes of $u$ satisfies
$$Cap_p(\{u=0\})=0.$$

\begin{proof}
  Assume that $u^+ \nequiv 0$ and let us show then that $u^-\equiv 0$.

First observe that $a(x,\xi)\xi = \Phi(x,\xi)$. This fact follows from the homogeneity of $\Phi$ and Euler's differentiation formula for homogeneous mappings.

By using $u^+$  as test function  in \eqref{1.1}   we deduce that
$$\int_\Omega \Phi( x,\nabla u^+) = \lam_1 \int_\Omega \rho |u^+|^p$$
and therefore $u^+$ is also an eigenfunction corresponding to $\lam_1$. It satisfies hence \eqref{1.1} and we get
\begin{equation*}
\begin{cases}
-\diver(a(x,\nabla u^+)) = \lam_1 \rho |u^+|^{p-1}, & \text{in }\Omega,\\
u^+\geq 0,\quad u^+ \nequiv 0 & \text{in }\Omega,\\
u=0 & \text{on }\partial\Omega.
\end{cases}
\end{equation*}
By the maximum principle as stated in Theorem \ref{SMP}, we deduce that $u^-\equiv 0$ and $Cap_p(\{u=0\})=0$.
\end{proof}
\end{thm}

The following result gives the simplicity of the first eigenvalue. It follows by using a Picone
type identity, see  \cite{AlHu,Anane,  KawL, Lind}. Whenever the eigenfunctions
associated to $\lam_1$ are regular enough, the following  Picone type
identity holds.

\begin{lema} \label{picone}
  Let $v>0$, $u\geq 0$ be two continuous functions in $\Omega$ differentiable a.e. Let us denote
\begin{align*}
&L(u,v)=\Phi(x,\nabla u)+(p-1)\Big(\frac{u}{v}\Big)^p\Phi(x,\nabla v)-\Big(\frac{u}{v}\Big)^{p-1}\langle a(x,\nabla v),\nabla u \rangle,  \\
&R(u,v)=\langle a(x,\nabla u),\nabla u \rangle - \langle a(x,\nabla v), \nabla\Big(\frac{u^p}{v^{p-1}}\Big)\rangle.
\end{align*}
Then, \textbf{(i)} $L(u,v)=R(u,v)$, \textbf{(ii)} $L(u,v)\geq 0$, \textbf{(iii)} $L(u,v)=0$ a.e. in $\Omega$ if and only if $u=cu$ for some $c\in\R$.
\end{lema}

For the $p-$Laplacian, the regularity of eigenfunctions is known and it is enough to use Picone's identity. For general operators the proof is the same assuming that regularity, and the full proof without this assumption can be found  in \cite{Kaw-Lu-Pra}.

Now, simplicity of the first eigenvalue can be proved with a standard argument by using the Picone's identity given in Lemma \ref{picone}.

\begin{thm} \label{simple}
  Let $u$, $v$ be two eigenfunctions corresponding to $\lam_1$. Then there exists $c\in \R$ such that $u = c v$.
\end{thm}
\begin{proof}
  Let $u,v$ be two eigenfunctions associated to $\lam_1$. We can assume that $u$ and $v$ are both positive in $\Omega$. We apply Lemma \ref{picone} to  the pair $u, v+\ve$ and obtain
  \begin{align*}
    0&\leq \int_\Omega L(u,v+\ve)dx = \int_\Omega R(u,v+\ve)dx \\
    &=\lam_1 \int_\Omega \rho(x) |u|^p dx - \int_\Omega \langle a(x,\nabla v), \nabla\Big(\frac{u^p}{v^{p-1}}\Big)\rangle dx.
  \end{align*}
  Since the function $\frac{u^p}{(v+\ve)^{p-1}}\in W^{1,p}(\Omega)$, it is admissible in the weak formulation of $v$. It follows that
  $$0\leq \int_\Omega L(u,v+\ve)dx \leq \lam_1 \int_\Omega \rho(x)|u|^p\big(1-\frac{v^{p-1}}{(v+\ve)^{p+1}}\big) dx.$$
  Letting $\ve \cf 0$, we obtain
  $$\int_\Omega L(u,v)dx =0,$$
  but then $L(u,v)=0$ and by Lemma \ref{picone}, there exists $c\in\R$ such that $u=cv$.
\end{proof}

The proof in the general case, when Lemma \ref{picone} is not true  a.e. in $\Omega$, is
quite more complex and can be found in \cite{Kaw-Lu-Pra}, Theorem 1.3.

\begin{thm}[\cite{Kaw-Lu-Pra}, Section 6.2] \label{teo.lam1}
Let $u_1$ be an eigenfunction corresponding to $\lam_1$, then $u_1$ does not changes sign on $\Omega$. Also, the first eigenvalue is simple, that is, any other eigenfunction $u$ associated to $\lam_1$ is a multiple of $u_1$.
\end{thm}

Next, we show that the first eigenvalue $\lambda_1$ is isolated in $\Sigma$. The key step in the proof of the isolation is the next result:
\begin{prop}\label{teo2}
Let $\lambda\in\Sigma$ and let $w$ be an eigenfunction corresponding to $\lam\neq \lam_1$. Then, $w$ changes sign on $\Omega$, that is $u^+ \neq 0$ and $u^- \neq 0$. Moreover, there exists a positive constant $C$ independent of $w$ and $\lambda$ such that
$$
|\Omega^+|\geq C\lam^{-\gamma}, \quad |\Omega^-|\geq C\lam^{-\gamma},
$$
where $\Omega^\pm$ denotes de positivity and the negativity set of $w$ respectively, $\gamma$ is a positive parameter, and $C$ depends on $N, p, \rho^+$ and the coercivity constant $\alpha$ in {\em (H2)}. Here, $\gamma = (N-p)/p$ if $p< N$, $\gamma=1$ if $p=N$, and $\gamma=(p-N)/N$ if  $p>N$.
\end{prop}

\begin{proof}
Let $w$ be an eigenfunction corresponding to  $\lam\neq \lam_1$ and let $u$ be an eigenfunction corresponding to  $\lam_1$.

Assume that $w$ does not changes sign on $\Omega$. We can assume that $w\geq 0$ and $u\geq 0$ in $\Omega$. For each $k \in\N$, let us truncate $u$ as follows:
$$
u_k(x):=\min\{u(x), k\}
$$
and for each $\ve>0$ we consider the function $u_k^p/(w+\ve)^{p-1}\in W^{1,p}_0(\Omega)$. We get
\begin{equation} \label{ecuac2}
\int_\Omega a(x,\nabla u)\nabla u -a(x,\nabla w)\nabla\Big(\frac{u_k^p}{(w+\ve)^{p-1}}\Big)  = \int_\Omega  \lam_1 \rho u^{p} -\lam \rho w^{p-1}\frac{u_k^p}{(w+\ve)^{p-1}}
\end{equation}
We claim that the integral in the left hand side in \eqref{ecuac2} is non-negative. Indeed, let
$\Phi$ be the potential function given by Proposition \ref{potential.f}. Then, as $\Phi$ is
$p-$homogeneous in the second variable we have (see \cite{Kaw-Lu-Pra}, p.19)
\begin{equation} \label{ecuac3}
\begin{split}
&a(x,\nabla u)\nabla u -a(x,\nabla w)\nabla\Big(\frac{u_k^p}{(w+\ve)^{p-1}}\Big)=\\
&p \Big\{ \Phi(x,\nabla u)+ (p-1) \Phi(x,\frac{u_k}{w+\ve} \nabla w) - a(x,\frac{u_k}{w+\ve} \nabla w)\nabla u_k \Big\}.
\end{split}
\end{equation}
By using the property that $\xi \mapsto \Phi(x,\xi)$ is convex, we easily deduce that \eqref{ecuac3} is nonnegative. Therefore, coming back to \eqref{ecuac2} we get
\begin{align} \label{ecuac4}
 \int_\Omega \lam_1 \rho u^{p} - \lam \rho w^{p-1}\frac{u_k^p}{(w+\ve)^{p-1}} \geq 0.
\end{align}
Since by the strong maximum principle for quasilinear operators (Theorem \ref{SMP}) the set $\{\tilde w=0\}$, where $\tilde w$ is the $p-$quasi continuous representative of $w$, is of measure zero then \eqref{ecuac4} is equivalent to
\begin{equation} \label{ecuac5}
 \int_{\{w>0\}} \lam_1 \rho u^{p} -\lam \rho w^{p-1}\frac{u_k^p}{(w+\ve)^{p-1}} \geq 0.
\end{equation}
Now, letting $\ve \cf 0$ and $k\cf \infty$ in (\ref{ecuac5}), we get
$$
(\lam_1-\lam) \int_\Omega \rho |u|^p \geq 0
$$
which is a contradiction. Therefore $w$ changes sign on $\Omega$.

The second part of the proof follows almost exactly as in the $p-$Laplacian case. Let us suppose first that $p< N$. In fact, as $w$ changes sign, we can use $w^+$ as a test function in the equation satisfied by $w$ to obtain
\begin{align*}
\int_\Omega a(x,\nabla w)\nabla w^+ &= \lambda \int_\Omega \rho |w|^{p-2} w w^+ \\
& = \lam \int_{\Omega^+} \rho |w|^p\\
&\le \lam\rho^+ \int_{\Omega^+} |w|^p\\
& \le \lam\rho^+ \|w^+\|_{L^{p^*}(\Omega)}^p |\Omega^+|^{p/(N-p)}\\
&\le \lam \rho^+ K_p |\Omega^+|^{p/(N-p)}\int_\Omega |\nabla w^+|^p,
\end{align*}
where $K_p$ is the optimal constant in the Sobolev-Poincar\'e inequality.

Now, by (H2), it follows that
$$
\int_\Omega a(x,\nabla w)\nabla w^+\ge \alpha \int_\Omega |\nabla w^+|^p.
$$
Combining these two inequalities, we obtain
$$
|\Omega^+|\ge \Big(\frac{\alpha}{K_p\lam \rho^+}\Big)^{(N-p)/p}.
$$
The estimate for $|\Omega^-|$ follows in the same way.

The remaining cases are similar: $p=N$ follows by using the Sobolev's inclusion $W_0^{1,N}(\Omega) \subset L^N(\Omega)$, and the case $p>N$ follows from Morrey's inequality.
\end{proof}

Now we are ready to prove the isolation of $\lam_1$.
\begin{thm} \label{teo3}
The first eigenvalue $\lam_1$ is isolated. That is, there exists $\delta>0$ such that $(\lam_1, \lam_1+\delta)\cap\Sigma =\emptyset$.
\end{thm}

\begin{proof}
Assume by contradiction that there exists a sequence $\lam_j\in\Sigma$ such that $\lam_j\to\lam_1$ as $j\to\infty$. Let $u_j$ be the associated eigenfunctions normalized such that
$$
\int_{\Omega}\rho |u_j|^p = 1.
$$
By (H2) it follows that the sequence $\{u_j\}_{j\in\N}$ is bounded in $W^{1,p}_0(\Omega)$ so, passing to a subsequence if necessary, there exists $u\in W^{1,p}_0(\Omega)$ such that
\begin{align*}
& u_j \cd u \qquad \mbox{weakly in } W^{1,p}_0(\Omega)\\
& u_j \cf u \qquad \mbox{strongly in } L^{p}(\Omega)
\end{align*}
Now, as the functional
$$
v\mapsto \int_\Omega \Phi(x,\nabla v)
$$
is weakly sequentially lower semicontinuous  (see \cite{Con}), it follows that $u$ is an eigenfunction associated to $\lam_1$.

Now, by Theorem \ref{teo.lam1}, we can assume that $u\ge 0$ and by Proposition \ref{teo2}
we have $|\{ u=0\}|>0$. But this is a contradiction to the strong maximum principle in
\cite{Kaw-Lu-Pra}, Theorem \ref{SMP}.
\end{proof}

As a consequence of Theorem \ref{teo3} it makes sense to define the second eigenvalue $\Lam_2$ as the infimum of the eigenvalues greater than $\lam_1$. Next, we show that this second eigenvalue $\Lam_2$ coincides with the second variational eigenvalue $\lam_2$. This result is known to hold for the $p-$Laplacian (see \cite{AT})  and we extended here for the general case \eqref{eps1}.
\begin{thm} \label{teo4}
Let $\lam_2$ be the second variational eigenvalue, and let $\Lam_2$ be defined as
$$
\Lam_2 = \inf\{\lam > \lam_1 \colon \lam\in\Sigma\}.
$$
Then $$\lam_2=\Lam_2.$$
\end{thm}

\begin{proof}
The proof of this Theorem follows closely the one in \cite{FBR} where the analogous result for the Steklov problem for the $p-$Laplacian is analyzed.

Let us call
$$
\mu = \inf\left\{\int_\Omega \Phi(x,\nabla u) \colon\|\rho u \|_{ L^p(\Omega) }^p=1  \textrm{ and } |\Omega^\pm| > c_{\lam_2}\right\},
$$
where  $c_{\lam_2} := C \lam_2^{-\gamma}$ and  $C, \gamma$ are given by Proposition \ref{teo2}.

If we take $u_2$ an eigenfunction of \eqref{eps1} associated with $\Lam_2$  such that $\|\rho u \|_{ L^p(\Omega) }^p=1 $, by Theorem \ref{teo2}, we have that $u_2$ is admissible in the variational characterization of $\mu$. It follows that $\mu\leq \Lam_2$. The proof will follows if we show that $\mu \geq \lam_2$. The inverse of $\mu$ can be written as
$$
\frac{1}{\mu} = \sup\left\{\int_\Omega \rho |u|^p \colon \int_\Omega \Phi(x,\nabla u) =1  \textrm{ and } |\Omega^\pm| > c_{\lam_2}\right\}.
$$
The supremum is attained by a function $w\in W^{1,p}_0(\Omega)$ such that $\int_\Omega \Phi(x,\nabla w) =1$ and $|\Omega^\pm| > c_{\lam_2}$. As $w^+$ and $w^-$ are not identically zero, if we consider the set
$$
C=span\{w^+,w^-\} \cap \{ u\in W^{1,p}_0(\Omega)\colon \|u\|_{W^{1,p}_0(\Omega)}=1\},
$$
then $\gamma(C)=2$. Hence, we obtain
\begin{equation} \label{autov1}
\frac{1}{\lam_2} \geq \inf_{u\in C} \int_\Omega \rho |u|^p
\end{equation}
but, as $w^+$ and $w^-$ have disjoint support, it follows that the infimum \eqref{autov1} can be computed by minimizing the two variable function
$$
G(a,b):=|a|^p\int_\Omega \rho |w^+|^p+|b|^p\int_\Omega \rho |w^-|^p
$$
with the restriction
$$
H(a,b):=|a|^p \int_\Omega \Phi(x,\nabla w^+) + |b|^p \int_\Omega \Phi(x,\nabla w^-) = 1.
$$

Now, an easy computation shows that
$$
\frac{1}{\lam_2}\ge \min\left\{ \frac{\int_\Omega \rho |w^+|^p}{\int_{\Omega} \Phi(x,\nabla w^+)}, \frac{\int_\Omega \rho |w^-|^p}{\int_{\Omega} \Phi(x,\nabla w^-)}\right\}.
$$

We can assume that the minimum in the above inequality is realized with $w^+$. Then, for $t>-1$ the fuction $w+tw^+$ is admissible in the variational characterization of $\mu$, hence if we denote
$$
Q(t) :=  \frac{\int_\Omega \rho |w+tw^+|^p}{\int_{\Omega} \Phi(x,\nabla w + t\nabla w^+)},
$$
we get
$$
0 = Q'(0) = p \int_\Omega \rho |w|^{p-2}ww^+ - \frac{p}{\mu} \int_\Omega a(x,\nabla w)\nabla w^+,
$$
therefore
$$
 \frac{\int_\Omega \rho |w^+|^p}{\int_{\Omega} \Phi(x,\nabla w^+)}=\frac{1}{\mu}
$$
and the result follows.
\end{proof}

In dimensions $N>1$ it is not known even in the $p-$Laplacian case whether $\lambda_2$ is isolated in $\Sigma$ or not or if $\Sigma$ is countable or not. So we cannot expect to obtain much more information in the general case \eqref{eps1}.

However, in the one dimensional problem $N=1$ it is known since the work of Fu{\v{c}}{\'{\i}}k and coauthors in \cite{Fucik} (see also the more recent works \cite{FBP, Walter}) that $\Sigma=\Sigma_{\text{var}}$. So now we generalize this fact to \eqref{eps1}. That is, we study
\begin{equation}\label{autov.eq}
\begin{cases}
-(a(x)|u'|^{p-2}u')' = \lam \rho(x) |u|^{p-2}u & \mbox{in } J:=(0,\ell)\\
u(0)=u(\ell)=0
\end{cases}
\end{equation}
where $0<\rho_-\le \rho(x)\le \rho_+$ and $0<\alpha\le a(x)\le \beta$ for some
constants $\rho_-, \rho_+, \alpha$ and $\beta$.

For the one dimensional $p-$Laplace operator in $J$ with zero Dirichlet boundary conditions, that is \eqref{autov.eq} with $a(x)=\rho(x)=1$, we denote by $\tilde \Sigma = \tilde \Sigma_{\text{var}} = \{\mu_k\}_{k\in\N}$ the spectrum given by
\begin{equation}\label{formuvar}
\mu_{k} = \inf _{C\in \mathcal{C}_k} \sup _{u\in C} \frac{\int_J |u'|^p\, dx}{\int_J |u|^p\, dx}.
\end{equation}

Here, all the eigenvalues and eigenfunctions can be found explicitly:

\begin{thm}[Del Pino, Drabek and Manasevich, \cite{DDM}]\label{draman}
The eigenvalues $\mu_{k}$ given by \eqref{formuvar} and their corresponding eigenfunctions $u_{k}$ on the interval $J$ are given by
$$
\mu_{k} = \frac{\pi_p^p k^p}{\ell^p},
$$
$$
u_{k}(x) = \sin_p(\pi_p kx/\ell).
$$
\end{thm}

The function $\sin_p(x)$ is the solution of the initial value problem
\begin{equation*}
 \begin{cases}
-(|u'|^{p -2} u')' =  |u|^{p -2}u\\
\; u(0)=0, \qquad u'(0)=1,
 \end{cases}
\end{equation*}
and is defined implicitly as
$$
x =  \int_0^{\sin_p(x)} \Big(\frac{p-1}{1-t^p}\Big)^{1/p} dt.
$$
Moreover, its first zero is $\pi_p$, given by
$$
\pi_p = 2 \int_0^1 \Big(\frac{p-1}{1-t^p}\Big)^{1/p} dt.
$$

In \cite{ACM}, problem \eqref{autov.eq} with $a\equiv1$ is studied and, among other
things, it is proved that any eigenfunction associated to $\lam_k$ has exactly $k$
nodal domains. As a consequence of this fact, in \cite{ACM} it is obtain the
simplicity of every variational eigenvalue.

The exact same proof of \cite{ACM} works in our case, and so we obtain the following:
\begin{thm} \label{teo.zero}
Every eigenfunction corresponding to the $k-$th eigenvalue $\lam_k$ has exactly $k-1$
zeroes. Moreover, for every $k$, $\lam_k$ is simple, consequently the eigenvalues are
ordered as $0<\lam_1<\lam_2 < \cdots < \lam_k \nearrow +\infty$.
\end{thm}

Now, using the same ideas as in \cite{FBP} is easy to prove that the spectrum of
\eqref{autov.eq} coincides with the variational spectrum. In fact, we have:
\begin{thm} \label{teo.espec}
$\Sigma=\Sigma_{var}$.
\end{thm}
\begin{proof}
The proof of this theorem is completely analogous to that of Theorem 1.1 in
\cite{FBP}.
\end{proof}

\section{Eigenvalue homogenization}

In this section, as an application of the results in Section 3, we analyze the convergence of the spectrum $\Sigma_\ve$ of problem
\begin{equation}\label{pro1}
\begin{cases}
-\diver(a_\ve(x,\nabla u)) = \lambda^\ve \rho_\ve |u|^{p-2}u & \text{on }\Omega\\
u = 0 & \text{on }\partial\Omega
\end{cases}
\end{equation}
to the spectrum $\Sigma$ of the limit problem
\begin{equation}\label{limit.prob}
\begin{cases}
-\diver(a(x,\nabla u)) = \lambda \rho |u|^{p-2}u & \text{on }\Omega\\
u = 0 & \text{on }\partial\Omega
\end{cases}
\end{equation}
under the assumption that $\A_\ve$ $G-$converges to $\A$ and that $\rho_\ve\stackrel{*}{\rightharpoonup}\rho$ in $L^\infty(\Omega)$. Moreover, we assume that $a_\ve(x,\xi)$ satisfies (H0)--(H8) uniformly.

The result in this section are not original, since they were obtained in \cite{Ch-dP} (for $\rho_\ve, \rho\equiv 1$ though). Nevertheless, the proof that we provide are much simpler than those in \cite{Ch-dP}.

In the linear case, it is well known (see \cite{Allaire-book}) that the $G-$convergence of the
operators implies the convergence of their spectra in the sense that the $k$th--eigenvalue
$\lam_k^\ve$ converges to the $k$th--eigenvalue of the limit problem.

We want to study the convergence of the spectrum in the non-linear case. We begin with a general result for bounded sequences of eigenvalues. This result was already proved in \cite{Con} but we include here a simpler proof for the reader's convenience.

Along the proofs by normalized eigenfunctions we understand that $\|u\|_p=1$.

\begin{thm} \label{teo-converg}
Let $\Omega\subset \R^N$ be bounded. Let $\lam^\ve\in \Sigma_\ve$ be a sequence of eigenvalues of problems \eqref{pro1} with $\{u^\ve\}_{\ve>0}$ associated normalized eigenfunctions.

Assume that the sequence of eigenvalues is convergent
$$
\lim_{\ve\to 0^+}\lam^{\ve} = \lam.
$$
Then, $\lam\in \Sigma$ and there exists a sequence $\ve_j\to 0^+$ such that
$$
u^{\ve_j} \cd u \textrm{ weakly in } W^{1,p}_0(\Omega)
$$
with $u$ a normalized eigenfunction associated to $\lam$.
\end{thm}

\begin{rem}\label{remk}
In most applications, we take the sequence $\lam^\ve$ to be the sequence of the $k$th--variational eigenvalue of \eqref{pro1}. In this case, it is not difficult to check that the sequence $\{\lam^\ve_k\}_{\ve>0}$ is bounded and so, up to a subsequence, convergent.

In fact, by using the variational characterization of $\lam_k^\ve$, \eqref{cont.coer.phi} and our assumptions on $\rho$ we have that
$$
\frac{\alpha}{\rho^+} \frac{\int_\Omega   |\nabla v|^p}{\int_{\Omega}  |v|^p}  \leq \frac{\int_\Omega \Phi_\ve(x,\nabla v)}{\int_{\Omega}  \rho_\ve |v|^p} \leq \frac{\beta}{\rho^-}\frac{\int_\Omega   |\nabla v|^p }{\int_{ \Omega}  |v|^p},
$$
therefore
$$
\frac{\alpha}{\rho^+} \mu_k \leq \lam_k^\ve \leq  \frac{\beta}{\rho^-} \mu_k
$$
where $\mu_k$ is the $k$th variational eigenvalue of the $p-$Laplacian.
\end{rem}

\begin{proof}
As $\lam_\ve$ is bounded and $u^\ve$ is normalized, by  (H2) it follows that the sequence $\{u^\ve\}_{\ve>0}$ is bounded in $W^{1,p}_0(\Omega)$.

Therefore, up to some sequence $\ve_j\to 0$, we have that
\begin{equation} \label{ecc}
\begin{split}
&u^{\ve_j} \cd u \quad \textrm{ weakly in } W^{1,p}_0(\Omega)\\
&u^{\ve_j} \cf u \quad \textrm{ strongly in } L^p(\Omega).
\end{split}
\end{equation}
with $u$ also normalized.

We define the sequence of functions $f_{\ve}:=\lam^\ve \rho_\ve |u^\ve|^{p-2} u^\ve$. By using the fact that $\rho_\ve \cd \rho$ *-weakly in $L^\infty(\Omega)$ together with \eqref{ecc} it follows that
$$
f_{\ve_j} \cd f:=\lam \rho|u|^{p-2}u \quad \textrm{ weakly in } L^p(\Omega)
$$
and therefore
$$
f_{\ve_j} \cf f \quad \textrm{ strongly in } W^{-1,p'}(\Omega).
$$

By  Proposition \ref{Gconv.ve} we deduce that $u^{\ve_j}$ converges weakly in $W^{1,p}_0(\Omega)$ to the unique solution $v$ of the homogenized problem
\begin{equation*}
\begin{cases}
-div(a(x,\nabla v)) =\lam \rho |u|^{p-2}u &\quad \textrm{ in } \Omega \\
v = 0  &\quad \textrm{ on } \partial \Omega.
\end{cases}
\end{equation*}
By uniqueness of the limit, $v=u$ is a normalized eigenfunction of the homogenized problem.
\end{proof}

\begin{rem}
In the case where the sequence $\lam^\ve$ is the sequence of the $k$th--variational eigenvalues of \eqref{pro1} it would be desirable to prove that it converges to the $k$th--variational eigenvalue of the homogenized problem \eqref{limit.prob} (see Remark \ref{remk}).

Unfortunately, our method only allow us to treat the first and second variational eigenvalues in the general setting. In the one dimensional case, one can be more precise and this fact holds true. See \cite{Ch-dP} for a general proof of this fact using the $\Gamma-$convergence method.
\end{rem}

\subsection{Convergence of the first and second eigenvalue}
The first eigenvalue of \eqref{pro1} is the infimum of the Rayleigh quotient
$$
\lam_1^\ve=\inf_{v \in W^{1,p}_0(\Omega)} \frac{\int_\Omega \Phi_\ve(x,\nabla v)}
{\int_\Omega \rho_\ve |v|^p}.
$$
In the following result we prove the convergence of $\lam_1^\ve$ when $\ve$ tends to zero. \begin{thm} \label{teo-1era-autof} Let be $\lam_1^\ve$ the first eigenvalue of \eqref{pro1} and $\lam_1$ the first eigenvalue of the limit problem \eqref{limit.prob}, then
$$
\lim_{\ve \cf 0} \lam_1^\ve = \lam_1.
$$
Moreover, if $u_1^\ve$ and $u_1$ are the (normalized) nonnegative eigenfunctions of \eqref{pro1} and \eqref{limit.prob} associated to $\lam_1^\ve$ and $\lam_1$ respectively, then
$$
u_1^\ve \cd u_1 \quad \text{weakly in } W^{1,p}_0(\Omega).
$$
\end{thm}

\begin{rem}
In \cite{Con} using the theory of convergence of monotone operators the authors obtain the conclusions of Theorem \ref{teo-1era-autof}. We propose here a simple proof of this result which exploits the fact that the first eigenfunction has constant sign.
\end{rem}

\begin{proof}
Let $u_1^\ve$ be the nonnegative normalized eigenfunction associated to $\lam_1^\ve$, the uniqueness of $u_1^\ve$ follows from Theorem \ref{teo.lam1}.

By Theorem \ref{teo-converg}, up to some sequence, $u_1^\ve$ converges weakly in $W^{1,p}_0(\Omega)$ to $u$, an eigenfunction of the homogenized eigenvalue problem associated to $\lam=\lim_{\ve\to 0}\lam_1^\ve$.

But then, $u$ is a nonnegative normalized eigenfunction of the homogenized problem \eqref{limit.prob} and so $u=u_1$. Therefore $\lam=\lam_1$ and the uniqueness imply that the whole sequences $\lam^\ve_1$ and $u_1^\ve$ are convergent.
\end{proof}

Now we turn our attention to the second eigenvalue. For this purpose we use the fact that eigenfunctions associated to the second variational eigenvalue of problems \eqref{pro1} and \eqref{limit.prob} have, at least, two nodal domains (cf. Proposition \ref{teo2}).
\begin{thm}\label{converg-2do-autov}
Let $\lam_2^\ve$ be the second eigenvalue of \eqref{pro1} and $\lam_2$ be the second eigenvalue of the homogenized problem \eqref{limit.prob}. Then
$$
\lim_{\ve \cf 0} \lam_2^\ve = \lam_2
$$
\end{thm}

\begin{proof}
Let $u_2$ be a normalized eigenfunction associated to $\lam_2$ and let $\Omega^{\pm}$ be the positivity and the negativity sets of $u_2$ respectively. By standard elliptic regularity theory, $\Omega^{\pm}$ are open sets. Now, the previous result about the positivity of the first eigenfunction implies that the restrictions of $u_2$ to $\Omega^{\pm}$ are the first eigenfunctions of the problem in those sets.

We denote by $u^\ve_\pm$ the first eigenfunction of \eqref{pro1} in $\Omega^\pm$ respectively. Extending $u_\pm^\ve$ to $\Omega$ by 0, these function have disjoint supports and therefore they are linearly independent in $W^{1,p}_0(\Omega)$.

Let $S$ be the unit sphere in $W^{1,p}_0(\Omega)$ and we define the set $C_2^\ve$ as
$$
C_2^\ve:=\textrm{span}\{u^\ve_+,u^\ve_-\}\cap S.
$$

Clearly $C_2^\ve$ is compact, symmetric and $\gamma(C_2^\ve)=2$. Hence,
$$
\lam_2^\ve = \inf_{C\in \Gamma_2} \sup_{v  \in C} \frac{\int_\Omega \Phi_\ve(x,\nabla v)}{\int_\Omega \rho_\ve |v|^p} \leq \sup_{v \in C_2^\ve} \frac{\int_\Omega \Phi_\ve(x,\nabla v)}{\int_\Omega \rho_\ve |v|^p}
$$

As $C_2^\ve$ is compact, the supremum is achieved for some $v^\ve \in C_2^\ve$ which can be written as
$$
v^\ve = a_\ve u^{\ve}_+ + b_\ve u^{\ve}_-
$$
with $a_\ve, b_\ve\in \R$ such that $|a_\ve|^p+|b_\ve|^p=1$. Since the functions $u^\ve_+$ and $u^\ve_-$ have disjoint supports, we obtain, using the $p-$homogeneity of $\Phi_\ve$ (see Proposition \ref{potential.f}),
$$
\lam_2^\ve \leq \frac{\int_\Omega \Phi_\ve(x,\nabla v^\ve)}{\int_\Omega \rho_\ve |v^\ve|^p}  = \frac{|a_\ve|^p \int_{\Omega^+} \Phi_\ve(x,\nabla u^\ve_+) + |b_\ve|^p \int_{\Omega^-} \Phi_\ve(x,\nabla u^\ve_-)}{\int_\Omega \rho_\ve |v^\ve|^p}
$$
Using the definition of $u^\ve_\pm$, the above inequality can be rewritten as
\begin{equation}\label{cota.ep}
\lam_2^\ve \leq \frac{|a_\ve|^p \lam_{1,+}^\ve\int_{\Omega^+}\rho_\ve |u^\ve_+|^p + |b_\ve|^p \lam_{1,-}^\ve\int_{\Omega^-}\rho_\ve |u^\ve_-|^p}{\int_\Omega \rho_\ve |v^\ve|^p}\le \max\{\lambda_{1,+}^\ve, \lambda_{1,-}^\ve\}
\end{equation}
where $\lam_{1,\pm}^\ve$ is the first eigenvalue of \eqref{pro1} in the nodal domain $\Omega^\pm$ respectively.

Now, using Theorem \ref{teo-1era-autof}, we have that $\lam_{1,\pm}^\ve \cf \lam_{1,\pm}$
respectively, where $\lambda_{1,\pm}$ are the first eigenvalues of \eqref{limit.prob} in the
domains $\Omega^\pm$ respectively.  Moreover, we observe that these eigenvalues
$\lam_{1,\pm}$ are both equal to the second eigenvalue $\lam_2$ in $\Omega$, therefore
from \eqref{cota.ep}, we get
$$
\lam_2^\ve \le \lam_2+\delta
$$
for $\delta$ arbitrarily small and $\ve$ tending to zero. So,
\begin{equation}\label{cota1}
\limsup_{\ve \cf 0} \lam_2^\ve \leq \lam_2
\end{equation}

On the other hand, suppose that $\lim_{\ve \cf 0} \lam_2^\ve = \lambda$ where $\lambda\in \Sigma_h$. We claim that $\lambda>\lambda_1$.

In fact, we have that $u_2^\ve \cd u$ in $W^{1,p}_0(\Omega)$ where $u$ is a normalized eigenfunction associated to $\lambda$. As the measure of the positivity and negativity sets of $u_2^\ve$ are bounded below uniformly in $\ve>0$ (see Proposition \ref{teo2}), we have that either $u$ changes sign or $|\{ u=0 \}|>0$. In any case, this implies our claim.

Then, as $\lambda>\lambda_1$ it must be $\lambda\geq \lam_2$. Then
\begin{equation}\label{cota2}
\lam_2 \leq \lambda =\lim_{\ve \cf 0} \lam_2^\ve
\end{equation}
Combining \eqref{cota1} and \eqref{cota2} we obtain the desired result.
\end{proof}

\subsection{Convergence of the full spectrum in the one dimensional case}
The goal of this subsection is to prove the following Theorem

\begin{thm}\label{teo.gral}
Let $N=1$ and assume that $\A_\ve$ $G-$converges to $\A$ and that $\rho_\ve\cd \rho$ weakly* in $L^\infty(I)$. For each $k\geq 1$ let $\lam_k^\ve$ be the $k$-th eigenvalue of \eqref{pro1}. Then we have that
$$
\lim_{\ve \cf 0} \lam_k^\ve = \lam_k,
$$
where $\lam_k$ the $k-$th eigenvalue of \eqref{limit.prob}.

Moreover, up to a subsequence, an eigenfunction $u_k^\ve$  associated to $\lam_k^\ve$ converges weakly in $W^{1,p}_0(I)$ to $u_k$, an eigenfunction associated to $\lam_k$.
\end{thm}
 The
main tool that allows us to prove that $\lambda=\lambda_k$ is Theorem \ref{teo.zero}
that says that any eigenfunction associated to the $k-$th eigenvalue of \eqref{pro1} has
exactly $k$ nodal domains.

Moreover, we need a refinement of this result, namely an estimate on the measure of
each nodal domain independent on $\ve$. This is the content of the next Lemma.

\begin{lema} \label{separa-ceros}
 Let $\lam_k^\ve$ be a eigenvalue of \eqref{pro1} with corresponding eigenfunction $u_k^\ve$. Let $\mathcal{N}=\mathcal{N}(k,\ve)$ be a nodal domain of $u_k^\ve$. We have that
$$|\mathcal{N}|>C$$
where $C=C(k)$ is a positive constant independent of $\ve$.
\end{lema}

\begin{proof}
We can write $\lam_k^\ve$ as
$$
\lam_k^\ve(I) = \lam_1^\ve(\mathcal{N}) = \inf_{u\in W^{1,p}_0(\mathcal{N})}  \frac{\int_\mathcal{N} a_\ve(x) |u'|^p}{\int_\mathcal{N} \rho_\ve(x) |u|^p},
$$
by our assumptions \eqref{cota.rho} we get
$$
\lam_k^\ve(I) \ge \frac{\alpha}{\rho_+} \mu_1(\mathcal{N}) =\frac{\alpha}{\rho_+}\frac{\pi_p^p}{|\mathcal{N}|^p}
$$
where $\mu_1(\mathcal N)$ is the first eigenvalue of the $p-$Laplacian on $\mathcal{N}$.
Moreover,
$$
\lam_k^\ve(I)\leq \frac{\beta}{\rho_-} \mu_k(I) = \frac{\beta}{\rho_-}\pi_p^p k^p.
$$
Combining both inequalities we get
$$
|\mathcal{N}|^p\geq \frac{\alpha}{\rho_+}\frac{\pi_p^p}{\lam_k^\ve(\Omega)}\geq \frac{\alpha}{\beta}\frac{\rho_-}{\rho_+} \frac{1}{k^p}
$$
and the result follows.
\end{proof}
Now we are ready to establish the main result of this section:

\begin{proof}[Proof of Theorem \ref{teo.gral}]
Let $u_k$ be a normalized eigenfunction associated to $\lam_k$ and according to Theorem \ref{teo.zero}, let $I_i$, $i=1,\ldots, k$ be the nodal domains of $u_k$.

We denote by $u^\ve_i$ the first eigenfunction of \eqref{pro1} in $I_i$ respectively. Extending $u_i^\ve$ to $I$ by 0, these function have disjoint supports
and therefore they are linearly independent in $W^{1,p}_0(I)$.

Let $S$ be the unit sphere in $W^{1,p}_0(I)$ and we define the set $C_k^\ve$ as
$$
C_k^\ve:=\textrm{span}\{u^\ve_1,\ldots, u^\ve_k\}\cap S.
$$

Clearly $C_k^\ve$ is compact, symmetric and $\gamma(C_k^\ve)=k$. Hence,
$$
\lam_k^\ve = \inf_{C\in \Gamma_k} \sup_{v  \in C} \frac{\int_I a_\ve(x)|v'|^p}{\int_I \rho_\ve |v|^p} \leq \sup_{v \in C_k^\ve} \frac{\int_I a_\ve(x) |v'|^p}{\int_I \rho_\ve |v|^p}
$$

As $C_k^\ve$ is compact, the supremum is achieved for some $v^\ve \in C_k^\ve$ which can be written as
$$
v^\ve = \sum_{i=1}^k a_i^\ve u^{\ve}_i
$$
with $a_i^\ve\in \R$ such that $\sum_{i=1}^k |a_i^\ve|^p=1$. Since the functions $u^\ve_i$ have non-overlapping supports, we obtain
$$
\lam_k^\ve \leq \frac{\int_I a_\ve(x)|{v^\ve} '|^p}{\int_I \rho_\ve |v^\ve|^p}  = \frac{ \sum_{i=1}^k |a_i^\ve|^p \int_{I_i} a_\ve(x) |{u_i^\ve}'|^p }{\int_I \rho_\ve |v^\ve|^p}
$$
Using the definition of $u^\ve_i$, the above inequality can be rewritten as
\begin{equation}\label{cota.epr}
\lam_k^\ve \leq \frac{ \sum_{i=1}^k|a^\ve_i|^p \lam_{1,i}^\ve\int_{I_i}\rho_\ve |u^\ve_i|^p}{\int_I \rho_\ve |v^\ve|^p}\le \max_{1\leq i \leq k} \{\lambda_{1,i}^\ve\}
\end{equation}
where $\lam_{1,i}^\ve$ is the first eigenvalue of \eqref{pro1} in the nodal domain $\Omega_i$ respectively.

Now, using that $\lam_{1,i}^\ve \cf \lam_{1,i}$ respectively , where $\lambda_{1,i}$ are the
first eigenvalues of \eqref{limit.prob} in the domains $I_i$ respectively (see Theorem 4.4,
\cite{FBPS}).  Moreover, we observe that these eigenvalues $\lam_{1,i}$ are all equal to the
$k-$th eigenvalue $\lam_k$ in $I$, therefore from \eqref{cota.epr}, we get
$$
\lam_k^\ve \le \lam_k+\delta
$$
for $\delta$ arbitrarily small and $\ve$ tending to zero. So
\begin{equation}\label{cota11}
\limsup_{\ve \cf 0} \lam_k^\ve \leq \lam_k.
\end{equation}
On the other hand, suppose that $\lim_{\ve \cf 0} \lam_k^\ve = \lambda$. By Lemma \ref{separa-ceros} the $k$ nodal domains of $u_k^\epsilon$ have positive measure independent of $\ve$. Then it must be $\lambda\geq \lam_k$. It follows that
\begin{equation}\label{cota22}
\lam_k \leq \lambda =\lim_{\ve \cf 0} \lam_k^\ve
\end{equation}
Combining \eqref{cota11} and \eqref{cota22} we obtain the desired result.
\end{proof}

\section*{Acknowledgements}

This work was partially supported by Universidad de Buenos Aires under grant 20020100100400 and by CONICET (Argentina) PIP 5478/1438.

%\bibliographystyle{amsplain}
%\bibliography{biblio}
\end{document}